\documentclass[11pt]{amsart}
\usepackage[T1]{fontenc}
\usepackage[utf8]{inputenc}
\usepackage{lmodern}
\usepackage[a4paper,margin=32mm]{geometry}
\usepackage{amsmath,amssymb,amsfonts,amsthm,mathtools}
\usepackage{enumitem}
\usepackage{hyperref}
\usepackage[T1]{fontenc}
\usepackage[utf8]{inputenc}
\usepackage{lmodern}
\usepackage{microtype}
\usepackage{amsmath,amssymb,amsthm,mathtools,mathrsfs}
\usepackage{geometry}
\usepackage{enumitem}
\usepackage{hyperref}
\usepackage[nameinlink,noabbrev]{cleveref}

\hypersetup{colorlinks=true,linkcolor=blue,citecolor=blue,urlcolor=blue}
\numberwithin{equation}{section}

\DeclareMathOperator{\spec}{spec}
\newtheorem{theorem}{Theorem}[section]
\newtheorem{lemma}[theorem]{Lemma}
\newtheorem{proposition}[theorem]{Proposition}

\theoremstyle{remark}
\newtheorem{remark}[theorem]{Remark}

\begin{document}
\title[Arithmetic means of Walsh--Fourier partial sums]{Almost Everywhere
Convergence of Arithmetic Means of Walsh--Fourier Partial Sums Along
Subsequences}
\author{ Ushangi Goginava}
\address{U. Goginava, Department of Mathematical Sciences \\
	United Arab Emirates University, P.O. Box No. 15551\\
	Al Ain, Abu Dhabi, UAE}
\email{zazagoginava@gmail.com; ugoginava@uaeu.ac.ae}
\date{}

\begin{abstract}
Let $S_m f$ denote the $m$-th partial sum of the Walsh-Fourier series of $f \in L^1$. For an increasing sequence $a=(a(n))_{n \geq 1}$ of positive integers, consider the arithmetic means

$$
\sigma_N f:=\frac{1}{N} \sum_{n=1}^N S_{a(n)} f .
$$

Gát proved in 2019 that $\sigma_N f \rightarrow f$ almost everywhere for every $f \in L^1$ under the growth condition

$$
a(n+1) \geq\left(1+\frac{1}{n^\delta}\right) a(n), \quad 0<\delta<\frac{1}{2} .
$$

We show that the same conclusion remains valid throughout the full range $0<\delta<1$.
\end{abstract}

\subjclass[2020]{42C10, 42A24}
\keywords{Walsh--Fourier series, subsequences of partial sums, arithmetic
means, almost everywhere convergence}
\maketitle

\section{Introduction}

Let $\omega =(\omega _{n})_{n\geq 0}$ be the Walsh--Paley system on $\mathbb{%
I}=[0,1)$, and let $S_{n}f$ denote the $n$th partial sum of the
Walsh--Fourier series of $f\in L^{1}(\mathbb{I})$. The classical Walsh
analogue of the Fej\'{e}r--Lebesgue theorem states that the full arithmetic
means 
\begin{equation*}
\frac{1}{N}\sum_{n=1}^{N}S_{n}f
\end{equation*}%
converge almost everywhere to $f$ for every $f\in L^{1}(\mathbb{I})$; see
Fine~\cite{Fine} and Schipp--Wade--Simon--P\'{a}l~\cite{SchippBook}. A
natural refinement, going back to Zalcwasser's work for the trigonometric
system~\cite{Zalcwasser}, asks how sparse a subsequence of partial sums may
be while the corresponding arithmetic means still reconstruct every
integrable function almost everywhere. For trigonometric Fourier series, G\'at 
\cite{GatConstr} proved that the arithmetic means of the partial sums taken
along a lacunary subsequence converge almost everywhere for every integrable
function. On the other hand, in \cite{GatAnn} he answered a question posed
by Zalcwasser  \cite{Zalcwasser} by constructing a subsequence and an
integrable function for which the arithmetic means diverge almost everywhere
along this subsequence. We also refer the reader to recent papers \cite%
{GogMukh,GogOniani} on the divergence of Fourier series along subsequences.

For a strictly increasing sequence of positive integers $a=(a(n))_{n\ge1}$,
define 
\begin{equation}
\sigma_N f:=\frac1N\sum_{n=1}^N S_{a(n)}f,\qquad N\in\mathbb{N}.
\label{eq:sigma-def}
\end{equation}
If $a$ is lacunary, then the almost everywhere convergence $\sigma_N f\to f$
for all $f\in L^1(\mathbb{I})$ was proved by G\'at~\cite{Gat2010}. Later,
G\'at showed in~\cite{Gat2019} that the same conclusion remains valid
whenever 
\begin{equation}
a(n+1)\ge \left(1+\frac1{n^{\delta}}\right)a(n)\qquad (n\in\mathbb{N})
\label{eq:growth-old}
\end{equation}
for some $0<\delta<1/2$.

The purpose of this note is to extend the conclusion of \cite{Gat2019} from the range $0< \delta<1 / 2$ to the full range $0<\delta<1$.

\begin{theorem}
\label{thm:main} Let $f\in L^1(\mathbb{I})$, and let $a=(a(n))_{n\ge1}$ be
an increasing sequence of positive integers such that 
\begin{equation}
a(n+1)\ge \left(1+\frac1{n^{\delta}}\right)a(n)\qquad (n\in\mathbb{N})
\label{eq:growth}
\end{equation}
for some $0<\delta<1$. Then 
\begin{equation*}
\lim_{N\to\infty}\frac1N\sum_{n=1}^N S_{a(n)}f(x)=f(x) \qquad\text{for
almost every }x\in\mathbb{I}. 
\end{equation*}
\end{theorem}

The proof follows the general strategy of~\cite{Gat2019}, but with a
different decomposition of the index set. Instead of grouping the averages
in intervals of length $N^{2\delta}$, we work on dyadic shells $2^m\le
N<2^{m+1}$ and split each shell into residue classes modulo $R_m\asymp
2^{m\delta}$. The growth assumption then recovers lacunarity inside each
residue class, while the prefactor $R_m/2^m$ leads to a summable series $%
\sum_m 2^{-m(1-\delta)}$, which is precisely why the argument works for
every $\delta<1$.

\section{Walsh preliminaries and a kernel decomposition}

We follow the notation of~\cite{Gat2019,SchippBook}. Write 
\begin{equation*}
\mathbb{I}=[0,1),\qquad x=\sum_{j=0}^{\infty}x_j2^{-j-1},\qquad
n=\sum_{j=0}^{\infty}n_j2^j, 
\end{equation*}
where $x_j,n_j\in\{0,1\}$. For dyadic rationals we choose the terminating
expansion. Dyadic addition is defined by 
\begin{equation*}
x\dotplus y:=\sum_{j=0}^{\infty}\lvert x_j-y_j\rvert2^{-j-1}. 
\end{equation*}
For $x\in\mathbb{I}$ and $m\in\mathbb{N}$, let 
\begin{equation*}
I_m(x):=\{y\in\mathbb{I}:y_0=x_0,\dots,y_{m-1}=x_{m-1}\},\qquad I_m:=I_m(0). 
\end{equation*}
Let $\mathcal{A}_m$ be the $\sigma$-algebra generated by the dyadic
intervals of rank $m$, and let $\mathbb{E}_m$ be the corresponding
conditional expectation.

The Rademacher functions are $r_j(x)=(-1)^{x_j}$, and the Walsh--Paley
functions are 
\begin{equation*}
\omega_n(x):=\prod_{j=0}^{\infty} r_j(x)^{n_j}. 
\end{equation*}
The Dirichlet and Fej\'er kernels are given by 
\begin{equation*}
D_n:=\sum_{k=0}^{n-1}\omega_k, \qquad K_n:=\frac1n\sum_{k=1}^n D_k. 
\end{equation*}
For $f\in L^1(\mathbb{I})$, the $n$th partial sum is 
\begin{equation*}
S_n f(y):=\sum_{k=0}^{n-1}\widehat f(k)\omega_k(y)  =\int_{\mathbb{I}}
f(x\dotplus y)D_n(x)\,dx. 
\end{equation*}
We also use the dyadic convolution 
\begin{equation*}
(f*g)(y):=\int_{\mathbb{I}}f(x\dotplus y)g(x)\,dx. 
\end{equation*}
It is well known that $S_{2^m}f=\mathbb{E}_m f$.

For $n\ge1$, set 
\begin{equation*}
|n|:=\max\{j\in\mathbb{N}:n_j\ne0\}, \qquad m(n):=\min\{j\in\mathbb{N}%
:n_j=1\}, 
\end{equation*}
and for $j\ge0$ define the upper truncation 
\begin{equation*}
n^{(j)}:=\sum_{k=j}^{\infty} n_k2^k. 
\end{equation*}
For an $L^2$ function $h$, let 
\begin{equation*}
\spec(h):=\{k\in\mathbb{N}:\widehat h(k)\ne0\} 
\end{equation*}
denote its Walsh spectrum.

For $n\ge1$ and $0\le i<|n|$, define 
\begin{equation}
\lambda_{n,i}:= 
\begin{cases}
1, & \text{if }m(n)\le i<|n|\text{ and }n_i=0, \\ 
0, & \text{otherwise,}%
\end{cases}
\label{eq:lambda-def}
\end{equation}
and 
\begin{equation}
d_{n,i}:=\lambda_{n,i}\,\omega_n\bigl(D_{2^i}-D_{2^{i+1}}\bigr).
\label{eq:dni-def}
\end{equation}
The next proposition gives the decomposition needed below.

\begin{proposition}
\label{prop:kernel-decomp} For every $n\ge1$, 
\begin{equation}
D_n = D_{2^{|n|+1}}-\omega_n D_{2^{m(n)}}+\sum_{i=0}^{|n|-1}d_{n,i}.
\label{eq:kernel-decomp}
\end{equation}
Consequently, 
\begin{equation}
S_n f =\mathbb{E}_{|n|+1}f-\omega_n\mathbb{E}_{m(n)}(f\omega_n)+\widetilde
S_n f, \qquad \widetilde S_n f:=\sum_{i=0}^{|n|-1} f*d_{n,i}.
\label{eq:partial-sum-decomp}
\end{equation}
\end{proposition}

\begin{proof}
A standard identity for the Walsh Dirichlet kernel is 
\begin{equation}
D_n=\omega_n\sum_{i=0}^{|n|}n_i\bigl(D_{2^{i+1}}-D_{2^i}\bigr).
\label{eq:dir-standard}
\end{equation}
Since $n_i=0$ for $i<m(n)$, $n_{|n|}=1$, and $n_i=1-\lambda_{n,i}$ for $%
m(n)\le i<|n|$, we obtain 
\begin{equation*}
\sum_{i=0}^{|n|}n_i\bigl(D_{2^{i+1}}-D_{2^i}\bigr)
= \sum_{i=m(n)}^{|n|}\bigl(D_{2^{i+1}}-D_{2^i}\bigr)
- \sum_{i=0}^{|n|-1}\lambda_{n,i}\bigl(D_{2^{i+1}}-D_{2^i}\bigr). 
\end{equation*}
The first sum telescopes to $D_{2^{|n|+1}}-D_{2^{m(n)}}$. Multiplying by $%
\omega_n$ and using that $\omega_n\equiv1$ on the support of $D_{2^{|n|+1}}$%
, namely on $I_{|n|+1}$, gives 
\begin{equation*}
D_n = D_{2^{|n|+1}}-\omega_n D_{2^{m(n)}}
+\sum_{i=0}^{|n|-1}\lambda_{n,i}\,\omega_n\bigl(D_{2^i}-D_{2^{i+1}}\bigr), 
\end{equation*}
which is exactly \eqref{eq:kernel-decomp}.

Convolving \eqref{eq:kernel-decomp} with $f$ yields 
\begin{equation*}
S_n f= S_{2^{|n|+1}}f-(f*(\omega_nD_{2^{m(n)}}))+
\sum_{i=0}^{|n|-1}f*d_{n,i}. 
\end{equation*}
Now $S_{2^{|n|+1}}f=\mathbb{E}_{|n|+1}f$, and a standard Walsh character calculation shows that 
\begin{equation*}
(f*(\omega_nD_{2^{m(n)}}))(y)=\omega_n(y)\,\mathbb{E}_{m(n)}(f\omega_n)(y). 
\end{equation*}
Hence \eqref{eq:partial-sum-decomp} follows.
\end{proof}

For the sequence $a=(a(n))_{n\ge1}$, put $A(n):=|a(n)|$ and define 
\begin{align}
\sigma_N^0 f&:=\frac1N\sum_{n=1}^N \mathbb{E}_{A(n)+1}f,
\label{eq:sigma0-def} \\
\rho_N f&:=\frac1N\sum_{n=1}^N \omega_{a(n)}\mathbb{E}_{m(a(n))}(f%
\omega_{a(n)}),  \label{eq:rho-def} \\
\widetilde\sigma_N f&:=\frac1N\sum_{n=1}^N \widetilde S_{a(n)}f.
\label{eq:sigma-tilde-def}
\end{align}
Then \eqref{eq:partial-sum-decomp} becomes 
\begin{equation}
\sigma_N f=\sigma_N^0 f-\rho_N f+\widetilde\sigma_N f.
\label{eq:sigma-splitting}
\end{equation}

We use the dyadic maximal operator 
\begin{equation*}
f^*(x):=\sup_{m\ge0}\mathbb{E}_m\lvert f\rvert(x). 
\end{equation*}
Since $\lvert \mathbb{E}_m g\rvert\le \mathbb{E}_m\lvert g\rvert$, we have 
\begin{equation}
\sup_{N\ge1}\lvert \sigma_N^0 f\rvert+\sup_{N\ge1}\lvert \rho_N f\rvert \le
2f^*.  \label{eq:max-first-pieces}
\end{equation}
We shall also use the classical weak-type estimate for the dyadic maximal
function (see, for instance,~\cite{SchippBook}) 
\begin{equation}
\lvert \{x\in\mathbb{I}:f^*(x)>\lambda\}\rvert \le \frac{C}{\lambda}\lVert
f\rVert_1 \qquad (\lambda>0).  \label{eq:dyadic-maximal-weak}
\end{equation}

\section{Frequency blocks}

For $n\ge1$ and $0\le i<|n|$, define the block 
\begin{equation}
B(n,i):=\bigl\{n^{(i+1)}+2^i+r:0\le r<2^i\bigr\}.  \label{eq:block-def}
\end{equation}

\begin{lemma}
\label{lem:block-form} Let $n\ge1$ and $0\le i<|n|$.

\begin{enumerate} [label=(\alph*)]

\item If $\lambda_{n,i}=0$, then $d_{n,i}=0$.

\item If $\lambda_{n,i}=1$, then 
\begin{equation}
d_{n,i}=-\omega_{n^{(i+1)}+2^i}D_{2^i}
=-\sum_{r=0}^{2^i-1}\omega_{n^{(i+1)}+2^i+r}.  \label{eq:dni-block}
\end{equation}
In particular, 
\begin{equation}
\spec(d_{n,i})=B(n,i)\subset [2^{|n|},2^{|n|+1}).  \label{eq:dni-spec}
\end{equation}

\item If $0\le i<j<|n|$ and $\lambda_{n,i}=\lambda_{n,j}=1$, then 
\begin{equation}
\max B(n,i)<\min B(n,j).  \label{eq:blocks-disjoint}
\end{equation}
Hence the nonzero blocks $B(n,i)$ are pairwise disjoint.
\end{enumerate}
\end{lemma}

\begin{proof}
Part~(a) is immediate from \eqref{eq:dni-def}. If $\lambda_{n,i}=1$, then $%
n_i=0$ and $n-n^{(i+1)}<2^i$. Hence $\omega_{n-n^{(i+1)}}\equiv1$ on the
support of both $D_{2^i}$ and $D_{2^{i+1}}$, so 
\begin{equation*}
d_{n,i}=\omega_{n^{(i+1)}}(D_{2^i}-D_{2^{i+1}}). 
\end{equation*}
Using $D_{2^{i+1}}=D_{2^i}+\omega_{2^i}D_{2^i}$, we get 
\begin{equation*}
d_{n,i}=-\omega_{n^{(i+1)}}\omega_{2^i}D_{2^i}
=-\omega_{n^{(i+1)}+2^i}D_{2^i}, 
\end{equation*}
which is \eqref{eq:dni-block}. Expanding $D_{2^i}=\sum_{r=0}^{2^i-1}\omega_r$
yields \eqref{eq:dni-spec}.

For part~(c), assume $i<j$ and $\lambda_{n,i}=\lambda_{n,j}=1$. Since $n_j=0$%
, 
\begin{equation*}
n^{(i+1)}=n^{(j+1)}+\sum_{k=i+1}^{j-1}n_k2^k \le
n^{(j+1)}+\sum_{k=i+1}^{j-1}2^k = n^{(j+1)}+2^j-2^{i+1}. 
\end{equation*}
Therefore 
\begin{equation*}
\max B(n,i)=n^{(i+1)}+2^{i+1}-1\le n^{(j+1)}+2^j-1< n^{(j+1)}+2^j=\min
B(n,j), 
\end{equation*}
as claimed.
\end{proof}

\begin{lemma}
\label{lem:block-stability} Let $n\ge1$ and $0\le i<|n|$. If $\phi$ is $%
\mathcal{A}_i$-measurable and $h\in L^2(\mathbb{I})$ satisfies $\spec%
(h)\subset B(n,i)$, then 
\begin{equation*}
\spec(\phi h)\subset B(n,i). 
\end{equation*}
\end{lemma}

\begin{proof}
Every $\mathcal{A}_i$-measurable function is a Walsh polynomial with
spectrum contained in $[0,2^i)$. Since multiplication of Walsh characters
corresponds to dyadic addition of the indices, multiplying by $\phi$ can
only change the lowest $i$ binary digits. The elements of $B(n,i)$ have the
form $n^{(i+1)}+2^i+r$ with $0\le r<2^i$, so changing the lowest $i$ digits
leaves the set $B(n,i)$ invariant. Hence $\spec(\phi h)\subset B(n,i)$.
\end{proof}

For $\lambda>0$, define the stopping time 
\begin{equation}
\nu_\lambda(x):=\inf\bigl\{m\in\mathbb{N}:\mathbb{E}_m(\lvert
f\rvert)(x)>\lambda\bigr\}, \qquad \inf\emptyset:=+\infty.
\label{eq:stopping-time}
\end{equation}
Since $\{\nu_\lambda>i\}=\bigcap_{m=0}^i\{\mathbb{E}_m(\lvert
f\rvert)\le\lambda\}\in\mathcal{A}_i$, the next lemma follows from %
\cref{lem:block-form,lem:block-stability}.

\begin{lemma}
\label{lem:Xn-spectrum} For $n\ge1$, set 
\begin{equation}
X_n^{(\lambda)}:=\sum_{i=0}^{|n|-1}\mathbf{1}_{\{\nu_\lambda>i\}}%
\,(f*d_{n,i}).  \label{eq:Xn-def}
\end{equation}
Then 
\begin{equation}
\spec\bigl(X_n^{(\lambda)}\bigr)\subset [2^{|n|},2^{|n|+1}),
\label{eq:Xn-shell}
\end{equation}
and the summands in \eqref{eq:Xn-def} are pairwise orthogonal in $L^2(%
\mathbb{I})$. In particular, 
\begin{equation}
\lVert X_n^{(\lambda)}\rVert_2^2 =\sum_{i=0}^{|n|-1}\lVert \mathbf{1}%
_{\{\nu_\lambda>i\}}\,(f*d_{n,i})\rVert_2^2.  \label{eq:Xn-orth}
\end{equation}
Moreover, $X_n^{(\lambda)}$ is $\mathcal{A}_{|n|+1}$-measurable and $\mathbb{%
E}_{|n|}X_n^{(\lambda)}=0$.
\end{lemma}

\begin{proof}
For each $i$, the function $f*d_{n,i}$ has spectrum contained in $\spec%
(d_{n,i})\subset B(n,i)$ by \eqref{eq:dni-spec}. Since $\mathbf{1}%
_{\{\nu_\lambda>i\}}$ is $\mathcal{A}_i$-measurable, %
\cref{lem:block-stability} gives 
\begin{equation*}
\spec\bigl(\mathbf{1}_{\{\nu_\lambda>i\}}(f*d_{n,i})\bigr)\subset
B(n,i)\subset [2^{|n|},2^{|n|+1}). 
\end{equation*}
By \cref{lem:block-form}(c), the nonzero blocks $B(n,i)$ are pairwise
disjoint, which yields both \eqref{eq:Xn-shell} and \eqref{eq:Xn-orth}. The
final assertion is immediate because a Walsh polynomial with frequencies in $%
[2^{|n|},2^{|n|+1})$ is $\mathcal{A}_{|n|+1}$-measurable and has zero $%
\mathcal{A}_{|n|}$-conditional expectation.
\end{proof}

The next estimate is the key input from G\'{a}t's paper \cite{Gat2010}.

\begin{lemma}[Stopped square estimate]
\label{lem:stopped-square} Let $f\in L^1(\mathbb{I})$, $n\ge1$, and $%
\lambda>0$. Then 
\begin{equation}
\sum_{i=0}^{|n|-1}\lVert \mathbf{1}_{\{\nu_\lambda>i\}}\,(f*d_{n,i})%
\rVert_2^2 \le C\lambda\lVert f\rVert_1.  \label{eq:stopped-square}
\end{equation}
\end{lemma}

\begin{lemma}[Doob's $L^2$ maximal inequality, see \cite{Durr}, Ch. 5]
\label{lem:doob} Let $(M_t,\mathcal{F}_t)_{t=0}^L$ be an $L^2$-martingale on
a probability space. Then 
\begin{equation*}
\Bigl\|\max_{0\le t\le L}\lvert M_t\rvert\Bigr\|_2^2\le 4\lVert
M_L\rVert_2^2. 
\end{equation*}
\end{lemma}

\section{A stopped maximal estimate}

Throughout this section, $a=(a(n))_{n\ge1}$ satisfies \eqref{eq:growth}. For
fixed $\lambda>0$, define 
\begin{align}
X_n&:=X_{a(n)}^{(\lambda)}=\sum_{i=0}^{A(n)-1}\mathbf{1}_{\{\nu_\lambda>i\}}%
\,(f*d_{a(n),i}),  \label{eq:Tlambda-def-a} \\
T_N^{(\lambda)}f&:=\frac1N\sum_{n=1}^N X_n, \qquad
T_\lambda^*f:=\sup_{N\ge1}\lvert T_N^{(\lambda)}f\rvert.
\label{eq:Tlambda-def}
\end{align}
By \cref{lem:Xn-spectrum,lem:stopped-square}, 
\begin{equation}
\lVert X_n\rVert_2^2\le C\lambda\lVert f\rVert_1 \qquad (n\ge1).
\label{eq:Xn-l2-bound}
\end{equation}

The next lemma extracts lacunarity inside a dyadic shell.

\begin{lemma}
\label{lem:lacunary-blocks} Fix $m\ge0$, and let 
\begin{equation*}
\kappa_\delta:=2^{1+\delta}\log 2+2, \qquad R_m:=\bigl\lceil \kappa_\delta
2^{m\delta}\bigr\rceil. 
\end{equation*}
If $\beta\ge1$ and $\beta+R_m\le 2^{m+1}$, then 
\begin{equation}
a(\beta+R_m)\ge 2a(\beta).  \label{eq:lacunary-step}
\end{equation}
Consequently, for each residue class $b\in\{0,1,\dots,R_m-1\}$, the
subsequence 
\begin{equation*}
a(b+jR_m),\qquad j\ge1, 
\end{equation*}
is lacunary with ratio at least $2$ as long as $b+jR_m\le2^{m+1}$. In
particular, the integers 
\begin{equation*}
A(b+jR_m)=|a(b+jR_m)| 
\end{equation*}
are then strictly increasing.
\end{lemma}

\begin{proof}
If $\beta+R_m\le 2^{m+1}$, then by \eqref{eq:growth}, 
\begin{equation*}
a(\beta+R_m) \ge
\prod_{t=\beta}^{\beta+R_m-1}\left(1+\frac1{t^\delta}\right)a(\beta) \ge
\left(1+2^{-(m+1)\delta}\right)^{R_m}a(\beta). 
\end{equation*}
Since $0<2^{-(m+1)\delta}\le1$, we have $\log(1+u)\ge u/2$ for $0<u\le1$.
Hence 
\begin{equation*}
R_m\log\left(1+2^{-(m+1)\delta}\right) \ge \kappa_\delta 2^{m\delta}\cdot
\frac12\,2^{-(m+1)\delta} = \kappa_\delta 2^{-1-\delta} \ge \log 2. 
\end{equation*}
Therefore 
\begin{equation*}
\left(1+2^{-(m+1)\delta}\right)^{R_m}\ge2, 
\end{equation*}
which proves \eqref{eq:lacunary-step}. The last assertion follows because $%
a(\beta+R_m)\ge2a(\beta)$ implies $|a(\beta+R_m)|\ge |a(\beta)|+1$.
\end{proof}

\begin{lemma}
\label{lem:T-weak} Assume \eqref{eq:growth} with $0<\delta<1$. Then, for
every $f\in L^1(\mathbb{I})$ and every $\lambda>0$, 
\begin{equation}
\lvert \{x\in\mathbb{I}:T_\lambda^*f(x)>\lambda\}\rvert \le \frac{C_\delta}{%
\lambda}\lVert f\rVert_1.  \label{eq:T-weak}
\end{equation}
\end{lemma}

\begin{proof}
Fix $m\ge0$ and $N$ with $2^m\le N<2^{m+1}$. Since 
\begin{equation*}
T_N^{(\lambda)}f=\frac1N\sum_{n=1}^{2^m}X_n+\frac1N\sum_{n=2^m+1}^{N}X_n, 
\end{equation*}
we have 
\begin{equation}
\lvert T_N^{(\lambda)}f\rvert^2 \le \frac{2}{2^{2m}}\lvert
\sum_{n=1}^{2^m}X_n\rvert^2 +\frac{2}{2^{2m}}\lvert
\sum_{n=2^m+1}^{N}X_n\rvert^2 =:A_m+B_{m,N}.  \label{eq:AB-splitting}
\end{equation}
We estimate the two terms separately.

\smallskip \noindent\emph{Estimate of $A_m$.} For $b=0,1,\dots,R_m-1$, set 
\begin{equation*}
J_{m,b}:=\bigl\{j\in\mathbb{N}_0:b+jR_m\in\{1,2,\dots,2^m\}\bigr\}. 
\end{equation*}
Then 
\begin{equation*}
\sum_{n=1}^{2^m}X_n =\sum_{b=0}^{R_m-1}\sum_{j\in J_{m,b}} X_{b+jR_m}. 
\end{equation*}
By Cauchy--Schwarz, 
\begin{equation*}
A_m\le C\frac{R_m}{2^{2m}}\sum_{b=0}^{R_m-1}\lvert \sum_{j\in
J_{m,b}}X_{b+jR_m}\rvert^2. 
\end{equation*}

 For fixed $b$, the spectra of the functions $X_{b+j R_m}$ lie in disjoint dyadic shells $\left[2^{A\left(b+j R_m\right)}, 2^{A\left(b+j R_m\right)+1}\right)$, because $A\left(b+j R_m\right)$ is strictly increasing by Lemma 4.1. Taking $L^1$-norms and using Parseval, %
\cref{lem:Xn-spectrum,lem:stopped-square,lem:lacunary-blocks}, we obtain 
\begin{align*}
\lVert A_m\rVert_1 &\le C\frac{R_m}{2^{2m}}\sum_{b=0}^{R_m-1} \Bigl\|%
\sum_{j\in J_{m,b}}X_{b+jR_m}\Bigr\|_2^2 \\
&= C\frac{R_m}{2^{2m}}\sum_{b=0}^{R_m-1}\sum_{j\in J_{m,b}}\lVert
X_{b+jR_m}\rVert_2^2 \\
&\le C\frac{R_m}{2^{2m}}\sum_{b=0}^{R_m-1}\sum_{j\in J_{m,b}}\lambda\lVert
f\rVert_1 \\
&\le C\frac{R_m}{2^m}\lambda\lVert f\rVert_1.
\end{align*}
Since $R_m\lesssim 2^{m\delta}$, 
\begin{equation}
\lVert A_m\rVert_1\le C2^{-m(1-\delta)}\lambda\lVert f\rVert_1.
\label{eq:Am-est}
\end{equation}

\smallskip \noindent\emph{Estimate of the tail term.} Set 
\begin{equation*}
B_m^*:=\sup_{2^m\le N<2^{m+1}}B_{m,N}. 
\end{equation*}
For each fixed residue class $b\in\{0,1,\dots,R_m-1\}$, list the indices in
the shell $\{2^m+1,\dots,2^{m+1}-1\}$ that are congruent to $b\pmod{R_m}$ in
increasing order: 
\begin{equation*}
n_{b,1}<n_{b,2}<\cdots<n_{b,L_{m,b}}. 
\end{equation*}
By construction, $n_{b,s+1}-n_{b,s}=R_m$, and therefore %
\cref{lem:lacunary-blocks} implies that the integers 
\begin{equation*}
k_{b,s}:=A(n_{b,s})=|a(n_{b,s})| 
\end{equation*}
are strictly increasing. For $1\le t\le L_{m,b}$, define 
\begin{equation*}
M_{b,t}:=\sum_{s=1}^t X_{n_{b,s}}, \qquad M_{b,0}:=0. 
\end{equation*}
For $2^m\le N<2^{m+1}$, there is an integer $t_b(N)\in\{0,\dots,L_{m,b}\}$
such that 
\begin{equation*}
\sum_{\substack{ 1\le s\le L_{m,b} \\ n_{b,s}\le N}}%
X_{n_{b,s}}=M_{b,t_b(N)}. 
\end{equation*}
Hence 
\begin{equation*}
\sum_{n=2^m+1}^{N}X_n=\sum_{b=0}^{R_m-1}M_{b,t_b(N)}, 
\end{equation*}
and another application of Cauchy--Schwarz yields 
\begin{equation*}
B_m^* \le C\frac{R_m}{2^{2m}}\sum_{b=0}^{R_m-1} \sup_{0\le t\le
L_{m,b}}\lvert M_{b,t}\rvert^2. 
\end{equation*}

Fix $b$. By \cref{lem:Xn-spectrum}, each increment $X_{n_{b,s}}$ has
spectrum contained in $[ 2^{k_{b,s}},2^{k_{b,s}+1}) $, so in particular 
\begin{equation*}
X_{n_{b,s}}\in L^2(\mathcal{A}_{k_{b,s}+1})\ominus L^2(\mathcal{A}%
_{k_{b,s}}). 
\end{equation*}
Since $k_{b,s+1}\ge k_{b,s}+1$, the filtration defined by 
\begin{equation*}
\mathcal{F}_{b,0}:=\{\emptyset,\mathbb{I}\}, \qquad \mathcal{F}_{b,t}:=%
\mathcal{A}_{k_{b,t}+1}\qquad (1\le t\le L_{m,b}) 
\end{equation*}
is increasing, $M_{b,t}$ is $\mathcal{F}_{b,t}$-measurable, and 
\begin{equation*}
\mathbb{E}\bigl(X_{n_{b,t+1}}\mid \mathcal{F}_{b,t}\bigr)=0. 
\end{equation*}
Thus $(M_{b,t},\mathcal{F}_{b,t})_{t=0}^{L_{m,b}}$ is an $L^2$-martingale.
By Doob's inequality, 
\begin{equation*}
\Bigl\|\sup_{0\le t\le L_{m,b}}\lvert M_{b,t}\rvert\Bigr\|_2^2 \le 4\lVert
M_{b,L_{m,b}}\rVert_2^2. 
\end{equation*}
The increments are pairwise orthogonal because their spectra lie in disjoint
dyadic shells, so 
\begin{equation*}
\lVert M_{b,L_{m,b}}\rVert_2^2 =\sum_{s=1}^{L_{m,b}}\lVert
X_{n_{b,s}}\rVert_2^2 \le C L_{m,b}\lambda\lVert f\rVert_1 
\end{equation*}
by \eqref{eq:Xn-l2-bound}. Therefore, 
\begin{align*}
\lVert B_m^*\rVert_1 &\le C\frac{R_m}{2^{2m}}\sum_{b=0}^{R_m-1} \Bigl\|%
\sup_{0\le t\le L_{m,b}}\lvert M_{b,t}\rvert\Bigr\|_2^2 \\
&\le C\frac{R_m}{2^{2m}}\sum_{b=0}^{R_m-1}\lVert M_{b,L_{m,b}}\rVert_2^2 \\
&\le C\frac{R_m}{2^{2m}}\sum_{b=0}^{R_m-1}L_{m,b}\lambda\lVert f\rVert_1 \\
&\le C\frac{R_m}{2^m}\lambda\lVert f\rVert_1.
\end{align*}
Hence 
\begin{equation}
\lVert B_m^*\rVert_1\le C2^{-m(1-\delta)}\lambda\lVert f\rVert_1.
\label{eq:Bm-est}
\end{equation}

Finally, \eqref{eq:AB-splitting} implies 
\begin{equation*}
\{T_\lambda^*f>\lambda\} \subset
\bigcup_{m=0}^{\infty}\{A_m+B_m^*>\lambda^2\}. 
\end{equation*}
By Chebyshev's inequality and \eqref{eq:Am-est}--\eqref{eq:Bm-est}, 
\begin{align*}
\lvert \{T_\lambda^*f>\lambda\}\rvert &\le
\frac1{\lambda^2}\sum_{m=0}^{\infty}\bigl(\lVert A_m\rVert_1+\lVert
B_m^*\rVert_1\bigr) \\
&\le \frac{C}{\lambda}\lVert f\rVert_1\sum_{m=0}^{\infty}2^{-m(1-\delta)}
\le \frac{C_\delta}{\lambda}\lVert f\rVert_1,
\end{align*}
because $0<\delta<1$. This proves \eqref{eq:T-weak}.
\end{proof}

\section{Proof of the main theorem}

Define the maximal operators 
\begin{equation*}
\widetilde M f:=\sup_{N\ge1}\lvert \widetilde\sigma_N f\rvert, \qquad
Mf:=\sup_{N\ge1}\lvert \sigma_N f\rvert. 
\end{equation*}
By \eqref{eq:sigma-splitting} and \eqref{eq:max-first-pieces}, 
\begin{equation}
Mf\le 2f^*+\widetilde M f.  \label{eq:M-split}
\end{equation}
We first show that $\widetilde M$ is of weak type $(1,1)$. Fix $\lambda>0
$, and let $T_N^{(\lambda)}f$ be defined by \eqref{eq:Tlambda-def}. On the
set $\{\nu_\lambda=\infty\}$, we have $\mathbf{1}_{\{\nu_\lambda>i\}}=1$ for
every $i$, hence 
\begin{equation*}
T_N^{(\lambda)}f=\widetilde\sigma_N f. 
\end{equation*}
Therefore 
\begin{equation*}
\{\widetilde M f>2\lambda\} \subset
\{\nu_\lambda<\infty\}\cup\{T_\lambda^*f>2\lambda\}. 
\end{equation*}
Using \eqref{eq:dyadic-maximal-weak} and \cref{lem:T-weak}, we obtain 
\begin{equation}
\lvert \{\widetilde M f>2\lambda\}\rvert \le \frac{C}{\lambda}\lVert
f\rVert_1.  \label{eq:Mtilde-weak}
\end{equation}
Combining \eqref{eq:M-split}, \eqref{eq:dyadic-maximal-weak}, and %
\eqref{eq:Mtilde-weak}, we conclude that 
\begin{equation}
\lvert \{Mf>4\lambda\}\rvert \le \frac{C}{\lambda}\lVert f\rVert_1.
\label{eq:M-weak}
\end{equation}
Thus the maximal operator associated with the averages \eqref{eq:sigma-def}
is of weak type $(1,1)$.

We now deduce almost everywhere convergence by a standard density argument.
Let $\mathcal{P}$ be the set of Walsh polynomials. If $P\in\mathcal{P}$,
then there exists $L\in\mathbb{N}$ such that $S_mP=P$ for every $m\ge L$.
Since $a(n)\to\infty$, it follows that 
\begin{equation}
\lim_{N\to\infty}\sigma_NP(x)=P(x) \qquad\text{for every }x\in\mathbb{I}.
\label{eq:poly-conv}
\end{equation}
Now let $f\in L^1(\mathbb{I})$, and choose $P_j\in\mathcal{P}$ with $\lVert
f-P_j\rVert_1\to0$. For each $j$ and for almost every $x\in\mathbb{I}$, 
\begin{equation*}
\limsup_{N\to\infty}\lvert \sigma_Nf(x)-f(x)\rvert \le M(f-P_j)(x)+\lvert
f(x)-P_j(x)\rvert, 
\end{equation*}
because $\sigma_NP_j(x)\to P_j(x)$ by \eqref{eq:poly-conv}. Hence, for every 
$\varepsilon>0$, 
\begin{equation*}
\Bigl\{x:\limsup_{N\to\infty}\lvert \sigma_Nf(x)-f(x)\rvert>2\varepsilon%
\Bigr\}
\subset \{M(f-P_j)>\varepsilon\}\cup\{\lvert f-P_j\rvert>\varepsilon\}. 
\end{equation*}
By \eqref{eq:M-weak} and Chebyshev's inequality, 
\begin{equation*}
\lvert \{M(f-P_j)>\varepsilon\}\rvert+\lvert \{\lvert
f-P_j\rvert>\varepsilon\}\rvert \le \frac{C}{\varepsilon}\lVert
f-P_j\rVert_1. 
\end{equation*}
Letting $j\to\infty$ gives 
\begin{equation*}
\lvert \Bigl\{x:\limsup_{N\to\infty}\lvert
\sigma_Nf(x)-f(x)\rvert>2\varepsilon\Bigr\}\rvert=0 \qquad (\varepsilon>0), 
\end{equation*}
which proves 
\begin{equation*}
\sigma_Nf(x)\to f(x) \qquad\text{for almost every }x\in\mathbb{I}. 
\end{equation*}
The proof of Theorem \ref{thm:main} is complete.

\begin{remark}
The argument does not reach the endpoint $\delta=1$. Indeed, in the shell $%
2^m\le N<2^{m+1}$ one needs a step length $R_m\asymp 2^{m\delta}$ to recover
lacunarity along residue classes. When $\delta=1$, this gives $R_m\asymp 2^m$%
, so the factor $R_m/2^m$ in the estimates of $A_m$ and $B_m^*$ no longer
decays. In particular, the quadratic subsequence $a(n)=n^2$ lies just beyond
the present method.
\end{remark}

\end{document}